\documentclass{amsart}

\title{Ultrafilters on singular cardinals of uncountable cofinality}

\author{James Cummings}

\address{Department of Mathematical Sciences\\Carnegie Mellon University\\Pittsburgh PA 15213-3890\\USA}
\email{jcumming@andrew.cmu.edu}

\thanks{Cummings was partially supported by the National Science Foundation, DMS-1500790.}

\author{Charles Morgan}
\email{charles\_jg\_morgan@email.com}

\usepackage{amssymb} 

\usepackage[dvipsnames]{xcolor} 

\DeclareMathOperator{\dom}{dom}
\DeclareMathOperator{\rge}{rge}
\DeclareMathOperator{\lh}{lh}
\DeclareMathOperator{\cf}{cf}
\DeclareMathOperator{\Ch}{Ch}
\DeclareMathOperator{\Spch}{Sp_\chi}
\DeclareMathOperator{\crit}{crit}
\DeclareMathOperator{\forces}{\Vdash}
\DeclareMathOperator{\OB}{OB}
\DeclareMathOperator{\Ult}{Ult}

\newtheorem{definition}{Definition}[section]
\newtheorem{lemma}[definition]{Lemma}
\newtheorem{claim}[definition]{Claim}
\newtheorem{subclaim}[definition]{Subclaim} 
\newtheorem{theorem}{Theorem} 
\newtheorem{fact}[definition]{Fact}

{\theoremstyle{definition}
  \newtheorem{remark}[definition]{Remark}
  \newtheorem*{notation}{Notation} 
}



\begin{document}

\begin{abstract} We prove that consistently there is a singular cardinal $\kappa$ of
  uncountable cofinality such that $2^\kappa$ is weakly inaccessible, and every
  regular cardinal strictly between $\kappa$ and $2^\kappa$ is the character of some
  uniform ultrafilter on $\kappa$.
\end{abstract}

\maketitle

\section{Introduction}

The {\em cardinal invariants of the continuum} are a family of cardinal numbers which
measure structural properties of the continuum.
Many of them are defined
from ${\omega}^\omega$ with the eventual domination ordering $\le^*$, $[\omega]^\omega$ with the almost inclusion
ordering $\subseteq^*$ or $\mathbb R$ with the null and meagre ideals,

Well known examples include:
\begin{itemize}
\item  $\mathfrak b$, the least size of an unbounded subset of $(\omega^\omega, \le^*)$.
\item  $\mathfrak d$, the least size of a cofinal subset of $(\omega^\omega, \le^*)$.
\item  $\mathfrak s$, the least size of a {\em splitting family} in $[\omega]^\omega$, that is a family
  $S$ such that for every $A \in [\omega]^\omega$ there is $B \in S$ with $A \cap B$ and $A \setminus B$ both
   infinite. 
 \item  $\mathfrak u$, the least size of a family in $[\omega]^\omega$ that generates a non-principal ultrafilter.
\end{itemize} 

 Assuming CH makes every reasonable cardinal invariant take the value $\omega_1$, while assuming MA makes every
 reasonable cardinal invariant take the value $2^{\aleph_0}$. A substantial research program in the set theory
 of the continuum has been to prove ZFC results which constrain the values of one or more cardinal invariants,
 either absolutely or in terms of other cardinal invariants, and
 complementary consistency results.

 One natural direction for generalisation is to replace $\omega$ by an uncountable regular cardinal $\kappa$.
 Some results generalise readily but new phenomena occur: notably the value of $\kappa^{<\kappa}$ is sometimes
 important, cardinal invariants associated with $\kappa$ and $\kappa^+$ can interact, and while the generalised
 invariants are typically defined using the co-bounded filter on $\kappa$ the club filter also plays a major role.

 We can also replace $\omega$ by a singular cardinal $\kappa$. Various issues arise here which are not present
 for regular $\kappa$: in general $2^{<\kappa}$ and $\kappa^{<\kappa}$ may not be equal, it's always true that
 $\kappa^{<\kappa} > \kappa$, the cobounded and club filters are only $\cf(\kappa)$-complete and the
 eventual domination and eventual inclusion orderings are less well-behaved. Nevertheless, Zapletal \cite{Zapletal} proved
 interesting results about the invariant ${\mathfrak s}(\kappa)$ in this setting.

 When $\kappa$ is an uncountable cardinal, the {correct}
 generalisation ${\mathfrak u}(\kappa)$ of the cardinal
 invariant $\mathfrak u$ involves {\em uniform} ultrafilters on $\kappa$, since a non-uniform ultrafilter on $\kappa$ is
 morally an ultrafilter on a smaller cardinal.
 \begin{definition}
   Let $\kappa$ be an infinite cardinal.
\begin{itemize}   
\item If $U$ is a uniform ultrafilter on $\kappa$, then a {\em base for $U$} 
  is a set $U' \subseteq U$ such that for every $X \in U$ there is $Y \in U'$ with
  $Y \subseteq^* X$.
\item The {\em character} $\Ch(U)$ of a uniform ultrafilter $U$ on $\kappa$ is the least
  size of a base for $U$.
\item  The {\em character spectrum} $\Spch(\kappa)$ of $\kappa$ is the set of characters
  of uniform ultrafilters on $\kappa$.
\item  ${\mathfrak u}(\kappa)$ is the minimum element of  $\Spch(\kappa)$. 
\end{itemize}
\end{definition} 

 It is not hard to see that $\kappa < {\mathfrak u}(\kappa) \le 2^\kappa$, so that for
 $\kappa$ singular and strong limit we can only obtain models with ${\mathfrak u}(\kappa) < 2^\kappa$
 by violating the Singular Cardinals Hypothesis. 

 There are several results about ${\mathfrak u}(\kappa)$ and $\Spch(\kappa)$ when $\kappa$ is singular strong limit
 in the literature:
\begin{itemize}
\item (Garti and Shelah \cite[Corollary 1.5]{GartiShelah}) Let $\kappa$ be supercompact, let GCH hold and let $\lambda < \kappa$ be regular.
  Then there are cardinal-preserving generic extensions in which $\cf(\kappa) = \lambda$, $2^\kappa$ is arbitrarily
  large and ${\mathfrak u}(\kappa) = \kappa^+$.
\item (Garti, Magidor and Shelah \cite[Theorem 9]{GartiMagidorShelah})
  Let $\kappa$ be strong,\footnote{The authors say supercompact but their proof uses only that $\kappa$ is strong
    with  measurable cardinals above it.}
  let GCH hold
  and let $\langle \mu_i : i < j \rangle$ be an increasing sequence
  of measurable cardinals above $\kappa$. Let $\langle \chi_i : i < j \rangle$ be an increasing sequence of
  regular cardinals above $\kappa$ with $\chi_i \le \mu_i < \chi_{i+1}$. Then there is a generic extension
  in which $\kappa$ is a singular strong limit cardinal of cofinality $\omega$, the cardinals
  $\chi_i$ remain regular, and $\{ \chi_i : i < j \} \subseteq \Spch(\kappa)$.
\item (Garti, Gitik and Shelah \cite{GartiGitikShelah}) It is consistent that $u_{\aleph_\omega} < 2^{\aleph_\omega}$
  with $\aleph_\omega$ strong limit.
\item (Gitik \cite{Gitik1})  It is consistent that a uniform ultrafilter over a singular cardinal can have singular character.
\end{itemize}
Gitik \cite{Gitik2, Gitik3} has also proved a number of interesting results about the related
 notion of ``strongly uniform ultrafilter'' and the related invariants. 

In this paper we extend the results of \cite{GartiMagidorShelah} to the situation where
the singular cardinal $\kappa$ has uncountable cofinality (see Theorems \ref{mainthm} and
 \ref{mainthm2} in Section \ref{mainsectionthm}).  
Our main  tool is a variant of Merimovich's
``extender-based Magidor-Radin forcing'' \cite{Carmi2011}, which has been modified 
 to exert  finer control over the cardinal arithmetic and PCF structure of the
 generic extension. 

 A key point is to construct certain PCF-theoretic scales in the extension,
 defined on reduced products of measurable cardinals where GCH holds and the
 corresponding reduced products of their successors. The arguments are somewhat
 parallel to those from \cite{GartiMagidorShelah} but there are new difficulties,
 in particular:
\begin{itemize} 
\item   By Silver's theorem and the subsequent work of Galvin and Hajnal, a severe failure  
  of GCH at a singular strong limit cardinal $\kappa$ of uncountable cofinality implies
  severe failure of GCH at almost every smaller cardinal. This is reflected in the structure
  of the forcing, and makes it harder to find  suitable sets of cardinals on which to define
  our scales. 
\item  The arguments of \cite{GartiMagidorShelah} use an extender-based forcing built from
   a single extender, and hinge on some analysis of the PCF structure in the corresponding extension
   due to Merimovich \cite{Carmi2007}. The PCF analysis is substantially harder 
   for us: there are many extenders involved, and there are 
   various difficulties whose root cause is that the forcing  conditions  are much more 
   complex objects than in the one-extender case. 
\end{itemize} 
 See the beginning of Section \ref{scales}
 for a more detailed discussion of the issues that arise in the scale construction.

Here is an outline of the rest of the paper:
\begin{itemize} 
\item In Section \ref{genufsection}, we review the construction of uniform ultrafilters
   with specified characters from appropriate scales.  
\item In Section \ref{backgroundsection}, we give some background on Radin forcing and
 the one-extender form of extender-based forcing, intended to motivate the extender-based
 Radin forcing of the following section. 
\item In Section \ref{EBF} we define a version of extender-based Radin forcing, and discuss
 its basic properties and its relationship with the construction of \cite{Carmi2011}. 
\item In Section \ref{itns}, we construct some finite iterated ultrapowers involving extenders,
 which will be useful in the scale analysis of the following section.  
\item In Section \ref{scales}, we construct a family of scales in the generic extension 
  by the forcing from Section \ref{EBF}. 
\item In Section \ref{mainsectionthm}, we state and prove our main results, Theorems
 \ref{mainthm} and \ref{mainthm2}.
\end{itemize}

\begin{notation}
  Our notation is mostly standard. The arguments involve various manipulations with sequences, and we use the following conventions:
\begin{itemize}
\item Sequences are generally written with either a bar or an arrow over them, for example $\bar u$ or $\vec \nu$. 
\item The concatenation of two sequences $\vec\sigma$ and $\vec\tau$ is written $\vec\sigma^\frown \vec\tau$. The result of prepending (resp.~appending) 
 an object $x$ to $\vec\sigma$ is $\langle x \rangle^\frown \vec\sigma$ (resp.~$\vec\sigma^\frown \langle x \rangle$). 
\item If $\vec \sigma$ is a sequence and $i \le \lh(\vec \sigma)$, then $\vec \sigma \restriction i$ is the restricted sequence
  $\langle \sigma_j : j < i \rangle$.
\item If $\vec \nu$ is a sequence of (possibly partial) functions on some domain $D$, and $d \subseteq D$, then
$\vec \nu \restriction d$ is the sequence of restricted functions $\langle \nu_i \restriction d : i < \lh(\vec \nu) \rangle$.
 In principle this could clash with the notation ``$\vec \sigma \restriction i$'' as above, but this will not happen here.
\item  Restriction has a higher precedence than concatenation, so that for example
 $\langle \kappa \rangle^\frown \vec U \restriction i$ is the concatenation of the sequences $\langle \kappa \rangle$ and
  $\vec U \restriction i$. 
\end{itemize} 
\end{notation}

\section{Generating ultrafilters} \label{genufsection}

We need some machinery for generating ultrafilters on singular cardinals. We use results from
\cite{GartiShelah} and \cite{GartiMagidorShelah},
which we sketch here to make this paper more self-contained.

\begin{definition} Let $\kappa$ be a regular cardinal, and let $U$ be a uniform ultrafilter $U$.
  An {\em almost-decreasing generating sequence for $U$} is a  $\subseteq^*$-decreasing sequence
  $\langle A_i : i < \theta \rangle$ such that $\{ A_i : i < \theta \}$ forms a base for $U$.
\end{definition}

It is easy to see that:
\begin{itemize}
\item  If $U$ has an almost-decreasing generating sequence, then it has such a sequence
  $\langle A_i : i < \theta \rangle$ such that  $\theta = \cf(\theta) > \kappa$. Moreover, in this situation $\theta = \Ch(U)$. 
\item  If $U$ has an almost-decreasing generating sequence then $U$ is $\kappa$-complete, in particular
  $\kappa$ is a measurable cardinal. 
\item  If $\kappa$ is measurable, $2^\kappa = \kappa^+$ and $U$ is a normal measure on $\kappa$, then
  $U$ has an almost decreasing generating sequence of length $\kappa^+$.
\end{itemize}

\begin{remark} It is possible to produce measures on $\kappa$ with almost decreasing generating
 sequences of a prescribed length. The basic idea is to start with $\kappa$ which is indestructibly
 supercompact and regular $\theta > \kappa^+$, iterate the ``long Prikry forcing'' (also known as
 ``long Mathias forcing'') at $\kappa$ for $\theta$ steps,
 and then do a delicate argument to produce a measure $U$ such that for many $i < \theta$
 we have that $U \cap V[G_i]$ is the measure which was used at stage $i$ and the generic subset
 added at stage $i$ is in $U$. See for example \cite{DzamonjaShelah} or \cite{5authors} for constructions of this type:
 the posets iterated in these papers are elaborations of long Prikry forcing, but the arguments
 work equally well for iterating long Prikry forcing.
\end{remark}

We recall the concept of a {\em scale} from PCF theory. We only need this concept in its simplest form.
\begin{definition} Let $\lambda$ be a singular cardinal with $\cf(\lambda) = \tau$, and let
  $\langle \lambda_i : i < \tau \rangle$ be an increasing sequence of regular cardinals which
  is cofinal in $\lambda$. A {\em scale of length $\nu$ in $\prod_{i < \tau} \lambda_i$}
  is a sequence $\langle g_\eta : \eta < \nu \rangle$ of functions 
  which is increasing and cofinal in $(\prod_{i < \tau} \lambda_i, <^*)$, where $<^*$
 is the eventual domination ordering.
\end{definition}
  
\begin{remark}  Typically the length $\nu$ of a scale as above is a regular cardinal,
  and this will always be the case in this paper. It is easy to see that $\lambda < \nu \le 2^\lambda$.
\end{remark}

The following result is a very mild generalisation of  \cite[Claim 4]{GartiMagidorShelah}
 and \cite[Theorem 1.4]{GartiShelah}
 \begin{lemma} \label{scalegeneration} 
   Suppose $\kappa$ is a singular cardinal such that $\cf(\kappa) = \rho$ and
   $2^\rho \le \kappa$. Let $\langle \mu_i : i < \rho \rangle$ be an increasing and cofinal sequence
   in $\kappa$ such that each $\mu_i$ is measurable and carries a measure $U_i$ which is
   generated by an almost decreasing sequence of regular length $\theta_i$. Let
   $\langle f_\alpha : \alpha < \sigma \rangle$ be a scale in $\prod_{i < \rho} \mu_i$
    with $\sigma$ regular and $\kappa < \sigma$, and let $\langle g_\beta : \beta < \tau \rangle$ be a scale  
   in $\prod_{i < \rho} \theta_i$ with $\tau$ regular and $\sigma \le \tau$.

   Then there exists a uniform ultrafilter $U$ on $\kappa$ such that $\Ch(U) = \tau$.
\end{lemma} 

\begin{proof} We may assume that $\rho < \mu_0$. 
Let $\mu_i^* = \sup_{i' < i} \mu_{i'}$ for $i < \rho$, and note that
$\mu_i^* < \mu_i$ because $i < \rho < \mu_0 \le \mu_i$ {and $\mu_i$, being measurable, is regular}.
Then the sequence $\langle \mu_i^* : i < \rho \rangle$ is continuous, increasing
and cofinal in $\rho$. Clearly $\mu_0^* = 0$ and $\mu_{i+1}^* = \mu_i$,
so the cardinal $\kappa$ is the union of pairwise disjoint non-empty intervals of the form
$[\mu_i^*, \mu_i)$ for $i < \rho$. 

For each $i < \rho$, fix an almost decreasing generating  sequence $\langle A^i_\eta: \eta < \theta_i \rangle$ 
{for $U_i$} such  that $A^i_\eta \subseteq [\mu_i^*, \mu_i)$ for all $\eta < \theta_i$.     
Fix a uniform ultrafilter $E$ on $\rho$. We use the data $\vec f$, $\vec g$, $\vec A$ and $E$ to define a
  uniform ultrafilter
  $U$ on $\kappa$ with a small generating set.

  Given $X \in E$, $\alpha < \sigma$ and $\beta < \tau$, let
  \[
  Y_{X, \alpha, \beta} = \bigcup_{i \in X} (A^i_{g_{\beta}(i)} \setminus f_{\alpha}(i)).
  \]
  Note that $Y_{X,\alpha,\beta} \cap [\mu_i^*, \mu_i) = A^i_{g_{\beta}(i)} \setminus f_{\alpha}(i)$
    for $i \in X$ and $Y_{X,\alpha,\beta} \cap [\mu_i^*, \mu_i) = \emptyset$ for $i \notin X$.

 \begin{claim}
   The sets $Y_{X,\alpha,\beta}$ form a filter base of size $\tau$,  which generates a uniform filter. 
 \end{claim}
 
 \begin{proof}
   Since $2^\rho \le \sigma \le \tau$, it is immediate that $\vert E \times \sigma \times \tau \vert = \tau$. 
   
   Let $n < \omega$, let $\gamma < \kappa$ and let $(X_k, \alpha_k, \beta_k)
   \in E \times \sigma \times \tau$ for $k  < n$.  Let
   $i \in \bigcap_{k < n} X_k$ with $\mu_i > \gamma$. Note that each
   of the sets $A^i_{g_{\beta_k}(i)} \setminus f_{\alpha_k}(i)$ is in $U_i$,
   and choose $\eta$ in their intersection with $\eta > \gamma$. Then
   clearly $\eta \in \bigcap_{k < n} Y_{X_k, \alpha_k, \beta_k}$.
 \end{proof}

 \begin{claim} 
   The sets $Y_{X,\alpha,\beta}$ generate an ultrafilter.
 \end{claim}
 
 \begin{proof} Let $Y \subseteq \kappa$. Either $\{ i : Y \cap \mu_i \in U_i \} \in E$ or 
   $\{ i : Y^c \cap \mu_i \in U_i \} \in E$, so replacing $Y$ by $Y^c$ if necessary we may assume that
   $X_0 \in E$, where $X_0 = \{ i : Y \cap \mu_i \in U_i \}$. For each $i \in X_0$, let $g(i) < \theta_i$
   be such that $A^i_{g(i)} \subseteq^* Y \cap \mu_i$, and note that
   $A^i_\eta \subseteq^* Y \cap \mu_i$ for all $\eta \ge g(i)$. Since 
   $\langle g_{\beta} : \beta < \tau \rangle$ is a scale, there exist $\beta < \tau$ and $i_0 < \rho$  such that
   $g(i) < g_{\beta}(i)$ for all $i \in X_0 \setminus i_0$.

   Let $X_1 = X_0 \setminus i_0$. For every $i \in X_1$ we have
   $A^i_{g_{\beta}(i)} \subseteq^* A^i_{g(i)} \subseteq^* Y \cap \mu_i$,
   so we may choose $f(i) < \mu_i$ such that $A^i_{g_{\beta}(i)}
   \setminus f(i) \subseteq Y \cap \mu_i$.  Since $\langle f_\alpha : \alpha <
   \sigma \rangle$ is a scale, there exist $\alpha < \sigma$ and $i_1$ with
   $i_0 < i_1 < \rho$ such that $f(i) < f_\alpha(i)$ for all $i \in X_1
   \setminus i_1$. If we let $X = X_1 \setminus i_1$ then $X \in E$
   and $A^i_{g_{\beta}(i)} \setminus f_\alpha(i) \subseteq Y \cap {\mu_i}$
   for all $i \in X$, so by definition $Y_{X,\alpha,\beta} \subseteq Y$.
\end{proof}  

  Let $U$ be the ultrafilter generated by the sets $Y_{X,\alpha,\beta}$.  
  From the proof of the last claim, we see that $Y \in U$ if and only if 
  $Y_{X,\alpha,\beta} \subseteq Y$ for some $X, \alpha, \beta$; the proof also shows that
  $\alpha$ and $\beta$ may be chosen arbitrarily large.

 \begin{claim}
   The character of $U$ is exactly $\tau$. 
 \end{claim}

 \begin{proof}
   Suppose for a contradiction that $U'$ is a base for $U$ with $\vert U' \vert < \tau$.
   Fix $X \in E$ and $\alpha < \sigma$, and find $Y' \in U'$ and $i_0 < \rho$ such that
   $Y' \setminus \mu_{i_0} \subseteq Y_{X,\alpha,\beta}$ for unboundedly many $\beta < \tau$. Find
   $X_0, \alpha_0, \beta_0$ such that $Y_{X_0, \alpha_0, \beta_0} \subseteq Y' \setminus \mu_{i_0}$, so that
   $Y_{X_0, \alpha_0, \beta_0} \subseteq Y_{X,\alpha,\beta}$ for unboundedly many $\beta < \tau$.

   For all $i$ with $i_0 < i < \rho$, find $D_i \subseteq A^i_{g_{\beta_0}(i)}$ such that
   $D_i \in U_i$ and $A^i_{g_{\beta_0}(i)} \setminus D_i$ is unbounded, and then $g(i) > g_{\beta_0}(i)$
   such that $A^i_{g(i)} \subseteq^* D_i$; note that if $\eta > g(i)$ then
   $A^i_\eta \subseteq^* D_i$, so that $A^i_{g_{\beta_0}(i)} \setminus A^i_\eta$ is unbounded.
   
   Choose $\beta > \beta_0$ such that $g <^* g_{\beta}$ and $Y_{X_0, \alpha_0, \beta_0} \subseteq Y_{X,\alpha,\beta}$.
   Choose $i \in X_0 \cap X$ such that $g(i) < g_{\beta}(i)$. Finally choose
   $\delta \in A^i_{g_{\beta_0}(i)} \setminus A^i_{g_{\beta}(i)}$ such that  $\delta > f_{\alpha_0}(i), f_\alpha(i)$.
   $Y_{X_0, \alpha_0, \beta_0} \cap [\mu_i^*, \mu_i) = A^i_{g_{\beta_0}(i)} \setminus f_{\alpha_0}(i)$ and
     similarly 
   $Y_{X_0, \alpha, \beta} \cap [\mu_i^*, \mu_i) = A^i_{g_{\beta}(i)} \setminus f_{\alpha}(i)$.
   So  $\delta \in Y_{X_0, \alpha_0, \beta_0}$ and $\delta \notin Y_{X,\alpha,\beta}$, contradicting the
   choice of $\beta$. 
\end{proof} 

   This concludes the proof of Lemma \ref{scalegeneration}.
 
\end{proof}

\section{Some background for the main construction} \label{backgroundsection}

As we mentioned in the introduction, the proof of our main result uses a form of extender-based Radin forcing.
In this short section we describe a simple form of Radin forcing (due in this version to Mitchell \cite{Mitchell}, building on
work of Radin \cite{Radin})
and a simple form of extender-based forcing (due in this version to Gitik and Merimovich \cite[Section 3]{CarmiMoti},
building on work
of Gitik and Magidor \cite{GitikMagidor}). The intention is to help the reader who is less familiar with this type of
forcing construction to see the wood for the trees in the construction of Section \ref{EBF}. We encourage the expert reader
to skip this section and go straight to Section \ref{EBF}. 

\subsection{Radin forcing}

Let $\kappa$ be a measurable cardinal, let $\rho < \kappa$ be regular and uncountable
and let $\vec U = \langle U_i : i < \rho \rangle$ be a sequence of normal measures on $\kappa$ which is
{\em Mitchell increasing},
 that is $\langle U_i : i < j \rangle \in \Ult(V, U_j)$ for all $j < \rho$.
Let $\bar u = \langle \kappa \rangle^\frown \vec U$. For each $i < \rho$ define a
measure $W_i = \{  X  : \langle \kappa \rangle^\frown \vec U \restriction i \in j_{U_i}(X) \}$,
and note that $W_i$ concentrates on sequences of the form
$\bar u' = \langle  \kappa' \rangle^\frown \langle U'_{i'} : i' < i \rangle$  
  where $\kappa' < \kappa$ and $\langle U'_{i'} : i' < i \rangle$  
  is a Mitchell increasing sequence of measures on $\kappa'$. Let $W = \bigcap_{i < \rho} W_i$,
  so that $W$ is a $\kappa$-complete filter.

  The associated Radin forcing has conditions of the form
  \[
  \langle  (\bar u^0, A^0), \ldots (\bar u^n, A^n) \rangle
  \]
  where $\bar u^n = \bar u$ and $A^n \in W$. For $k < n$,
  $\bar u^k$ is a typical object for some measure $W_i$ of the sort described above.
  If $\lh(\bar u^k) = 1$ then $A^k = \emptyset$, otherwise $A^k$ is a large set for the
  filter derived from $\bar u^k$ in the same way that $W$ was derived from $\bar u$. 
  The sequence of cardinals $({\bar u}^k)_0$ is strictly increasing with $k$.

  A condition can be extended by performing a finite series of ``elementary'' extensions.
  One type of elementary extension is simply to shrink some $A^i$. The other
  is to interpolate a new pair $(\bar v, B)$ where for some $i$ we have
  $\bar v \in A_i$, $B \subseteq V_{v_0} \cap A_i$, and $(u_{i-1})_0 < v_0$ in the case when $i > 0$. 
  The construction of the filter derived from $\bar u^i$ when $\lh(\bar u^i) > 1$ assures that there is a large
  set of candidates for $v$. 

  The generic object for this forcing is a $\rho$-sequence $\langle {\bar u}(i) : i < \rho \rangle$
  where the sequence $\langle {\bar u}(i)_0 : i < \rho \rangle$ is increasing, continuous and cofinal in $\kappa$.
  A condition $\langle (\bar u^0, A^0), \ldots (\bar u^n, A^n) \rangle$ carries the information that $\bar u^i$
  must appear on the generic sequence, and that the remaining points on the generic sequence must be drawn from the appropriate
  large set $A^i$. The forcing is $\kappa^+$-cc and satisfies a version of the Prikry lemma, asserting that
  any question can be decided by shrinking large sets.  The forcing preserves cardinals but changes many cofinalities. 

  We note a point which is salient later for the forcing of Section \ref{EBF}. Having $n > 0$ and
  $\lh(\bar u^0) = 1 + \eta$, for some $\eta$ with $0 < \eta < \rho$, is not enough on its own to ensure that
  ${\bar u}^0$ appears as ${\bar u}(\omega^\eta)$ on the generic sequence:
  although the filter derived from $\bar u^0$ concentrates on shorter
  sequences, $A^0$ may contain sequences of length at least $1 +  \eta$.
  By shrinking $A^0$ to eliminate such sequences we may obtain
  a condition which forces ${\bar u}(\omega^\eta)$ to be ${\bar u}^0$.
   
\begin{remark} The forcing we described here is a very simple special case of Mitchell's forcing from \cite{Mitchell},
 which (in common with other forms of Radin forcing)  permits the defining sequence of measures to be much longer than
 the common critical point. 
\end{remark} 

  \subsection{Extender-based forcing with one extender} \label{one-extender}

  Let $j: V \rightarrow M$ be an embedding with $\crit(j) = \kappa$ and ${}^\kappa M \subseteq M$,
  and let $\lambda$ be a cardinal with $\kappa^+ \le \lambda < j(\kappa)$.

  Let $d \in [\kappa, \lambda)$
    with $\kappa \in d$  and $\vert d \vert \le \kappa$.
    Define 
    $E(d) = \{ X : (j \restriction d)^{-1} \in j(X) \}$. $E(d)$ is a measure and
    concentrates on the set of {\em $d$-objects},
    where a $d$-object is an order-preserving partial function $\nu$ from $d$ to $\kappa$ such that
    $\kappa \in \dom(\nu)$ and $\vert \dom(\nu) \vert \le \nu(\kappa) < \kappa$. 

    If $\mu$ and $\nu$ are $d$-objects then $\mu < \nu$ if and only if $\dom(\mu) \subseteq \dom(\nu)$
    and $\mu(\alpha) < \nu(\alpha)$ for all $\alpha \in \dom(\mu)$. A {\em $d$-tree} is
    a tree $T$ of finite increasing sequences of $d$-objects, such that for every node $\vec \mu \in T$ the
    set  $\{ \nu : {\vec \mu}^{{\hskip2pt}\frown} \langle \nu \rangle \in T \}$ is $E(d)$-large.   

    If $d \subseteq d'$ then it is easy to see that the map $\nu \mapsto \nu \restriction d$ is
    a map from the set of $d'$-objects to the set of $d$-objects, and projects $E(d')$ to $E(d)$. 
    If $T$ is a $d'$-tree then we abuse notation by writing
    $T \restriction d = \{ \vec \mu \restriction d : \vec \mu \in T \}$.

    Conditions in the associated \emph{extender based forcing} are pairs $(f, A)$ where
    $f$ is a function with $\dom(f) = d$ for some set $d$ as above, $f(\alpha)$ is a finite increasing sequence
    of elements of $\kappa$ for each $\alpha \in \dom(f)$, and $A$ is a $d$-tree.

  A condition can be extended by performing a finite series of ``elementary'' extensions.
  One type of elementary extension is to extend $(f, A)$ to $(f', A')$ where
  $f' \restriction \dom(f) = f$ and $A' \restriction \dom(f) \subseteq A$.
  The other is to choose some $\langle \nu \rangle \in A$ such that
  $f(\alpha)^\frown \langle \nu(\alpha) \rangle$ is increasing for all $\alpha \in \dom(\nu)$,
  replace $f(\alpha)$ by $f(\alpha)^\frown \langle \nu(\alpha) \rangle$ for $\alpha \in \dom(\nu)$,
  and replace $A$ by $A_{\langle \nu \rangle} = \{ \vec \mu : \langle \nu \rangle^\frown \vec \mu \in A \}$.

  The generic object for this forcing has the form $\langle f_\alpha: \kappa \le \alpha < \lambda \rangle$ 
  where each $f_\alpha$ is an increasing $\omega$-sequence and is cofinal in $\kappa$.
  The forcing is $\kappa^{++}$-cc and satisfies a version of the Prikry lemma, asserting that
  any question can be decided by forming an elementary extension of the first type decribed above.   The forcing adds no bounded subsets of $\kappa$,
  preserves all cardinals, and changes the cofinality of $\kappa$ to $\omega$.

  \begin{remark}  The \emph{extender based Radin forcing} which we describe in the next section is a
    common generalisation of
    the two forcings we have just described. It changes the cofinality of $\kappa$ to $\rho$ while adding
    $\lambda$ many cofinal $\rho$-sequences. This kind of result was first achieved by Segal \cite{Mirithesis},
    with an extender-based Magidor forcing.
    \end{remark} 

\begin{remark} 
  The forcing from Section \ref{one-extender} is closely related to a forcing of Gitik and Magidor \cite{GitikMagidor}.
  The main difference is that Gitik and Magidor's forcing is based on a Rudin-Keisler directed sequence of
  ultrafilters $\langle U_\nu : \kappa \le \nu < \lambda \rangle$, where $U_\nu = \{ X \subseteq \kappa : \nu \in j(X) \}$.
  A condition in their forcing is of the form $(f, A)$
  where $f$ is as above and $A$ is a tree of finite increasing sequences of elements of $\kappa$
  having $U_\nu$-large branching for a particular ``maximum coordinate'' $\nu = mc(\dom(f)) \in \dom(f)$; when
  $\langle \beta \rangle \in A$ is used to extend the condition, ``projected'' versions of $\beta$ are added at a
  certain set of fewer than $\kappa$ coordinates in $\dom(f)$. In the forcing we described here there
  is no need for the maximum coordinate $\nu$, instead each $d$-object chooses where its values are to be added,
  and the role of the maximum coordinate $\nu$ in generating a suitable measure is played by $(j \restriction d)^{-1}$.
 \end{remark}

\section{Extender-based Radin forcing} \label{EBF}

Let GCH hold. 
Let $\rho$, $\kappa$ and $\lambda$ be cardinals such that $\rho < \kappa < \lambda$ and:
\begin{enumerate}
\item  $\rho$ is regular and uncountable.
\item  $\lambda$ is an inaccessible limit of measurable cardinals, and is the least
  such cardinal greater than $\kappa$.
\item There exists a sequence of extenders
$\vec E = \langle E_i : i < \rho \rangle$ such that each $E_i$
  witnesses that $\kappa$ is $\lambda$-strong and has ${}^\kappa \Ult(V, E_i) \subseteq \Ult(V, E_i)$, and
 the sequence is {\em Mitchell increasing} in the sense that  $\langle E_i : i < j \rangle \in \Ult(V, E_j)$ for all $j < \rho$.
\end{enumerate}
  We note that it is straightforward to build a sequence $\vec E$ as above if $\kappa$ is $(\lambda+1)$-strong.

 We will describe an extender-based forcing which preserves all cardinals and forces that $\cf(\kappa) = \rho$ and
$2^\kappa = \lambda$. The key point will be that for every $V$-measurable cardinal $\mu$ with
$\kappa < \mu < \lambda$, the generic extension will contain scales that can be fed into the
machinery of Lemma \ref{scalegeneration} to produce a uniform ultrafilter on $\kappa$ with character $\mu^+$.   

Let $h: \kappa+1 \rightarrow \lambda+1$ be the function which maps $\alpha$
to the least inaccessible limit of measurable cardinals above
$\alpha$, and note that:
\begin{enumerate}
\item $h \restriction \kappa$ is a function from $\kappa$ to $\kappa$.
\item $h(\kappa) = \lambda$.
\item $j_E(h)(\kappa) = \lambda$ for any extender
  $E$ witnessing that $\kappa$ is $\lambda$-strong.
\end{enumerate}

We will use $\vec E$ to build a version $\mathbb P$ of the extender-based Radin
forcing ${\mathbb P}_{\vec E, \lambda}$ of Merimovich \cite{Carmi2011}.
For more details about the relationship between $\mathbb P$ and ${\mathbb P}_{\vec E, \lambda}$,
see Remark \ref{changes} at the end of this section.
Our forcing is
designed to exert finer control over cardinal arithmetic and scales in the generic extension.
We will use several ideas from \cite{Carmi2011}, in particular
\cite[Lemma 4.10]{Carmi2011} and \cite[Lemma 4.11]{Carmi2011}
afford an analysis of dense open sets which will play a critical role.

Here is an overview of the forcing $\mathbb P$.
\begin{itemize}
\item For each $\alpha$ with $\kappa \le \alpha < \lambda$, $\bar \alpha = \langle \alpha \rangle^\frown \vec E$. The intention is
  that $\bar\alpha$ will be a coordinate, to which the forcing will associate a certain $\rho$-sequence of elements of $V_\kappa$.
\item $\mathfrak D = \{ \bar \alpha : \kappa \le \alpha < \lambda \}$. To each non-empty $d \subseteq {\mathfrak D}$
  with $\vert d \vert \le \kappa$ and each $\xi < \rho$
  we associate a function $mc_\xi(d)$ with domain $j_{E_\xi}[d]$, defined by the equation
  $mc_\xi(d)(j_{E_\xi}(\bar\alpha)) = \langle \alpha \rangle^\frown \vec E \restriction \xi$.
  Since $mc_\xi(d) \in \Ult(V, E_\xi)$ by our hypotheses,
  we may define a measure $E_\xi(d) = \{ X : mc_\xi(d) \in j_{E_\xi}(X) \}$ and a filter $E(d) = \bigcap_{\xi < \rho} E_\xi(d)$.
\item  $E_\xi(d)$ concentrates on a set $\OB(d) \subseteq V_\kappa$ of {\em $d$-objects} which resemble $mc_\xi(d)$:
  in particular if $\nu$ is a $d$-object then $\bar \kappa \in \dom(\nu) \subseteq d$, $\nu(\bar\alpha)$ is a sequence
  consisting of an ordinal
  in the interval $[\nu(\bar\kappa)_0, h(\nu(\bar\kappa)_0))$
    followed by a Mitchell increasing sequence of extenders (which does not depend on $\bar\alpha$)  of some length $\xi < \rho$,
  each extender has critical point $\nu(\bar\kappa)_0$,
  $\vert \dom(\nu) \vert \le \nu(\bar \kappa)_0$, and $\nu$ is {\em order preserving} in the sense that
  if $\alpha < \beta$ then $\nu(\bar\alpha)_0 < \nu(\bar\beta)_0$.
\item   Merimovich \cite{Carmi2011}  uses the term {\em extender sequence} both for
  sequences consisting of extenders (such as $\vec E$), and for sequences consisting of an ordinal followed
  by a sequence of extenders (such as the values assumed by a $d$-object).
  To avoid any confusion we will reserve the term 
  {\em extender sequence} for sequences consisting of extenders, consistent with usage in
  inner model theory, and will use the term {\em tagged extender sequence} for sequences
  consisting of an ordinal followed by a sequence of extenders.
  Tagged extender sequences are ordered by comparing their initial entries.
  The {\em order}  $o(\vec e)$ of a extender sequence $\vec e$ is just its length,  
  the {\em order} $o(x)$ of a tagged extender sequence $x = \langle\tau\rangle^\frown \vec e$ is
  $o(\vec e)$, and the {\em order} $o(\mu)$ of a $d$-object $\mu$ is the order of the tagged extender sequence $\mu(\bar\kappa)$
  (which is also the order of $\mu(\bar\alpha)$ for all $\bar\alpha \in \dom(\mu)$).
\item  The $d$-objects are ordered by $\mu < \nu$ iff $\dom(\mu) \subseteq \dom(\nu)$,
  $h(\mu(\bar\kappa)_0) < \nu(\bar\kappa)_0$,   and
  $\mu(\bar\alpha)_0 < \nu(\bar\alpha)_0$ for all $\bar\alpha \in \dom(\mu)$. 
\item A condition $p$ is a non-empty finite sequence whose last entry is denoted $p_{\rightarrow}$,
  where  $p_\rightarrow$ is a pair $(f^{p_\rightarrow}, A^{p_\rightarrow})$.
  Here $f^{p_\rightarrow}$ is a function such that  $\bar \kappa \in \dom(f^{p_\rightarrow}) \subseteq {\mathfrak D}$,
  $\vert \dom(f^{p_\rightarrow}) \vert \le \kappa$,
  and $f^{p_\rightarrow}(\bar\alpha)$ is a finite increasing sequence of tagged extender sequences  whose orders
  are less than $\rho$ and are non-increasing, while $A^{p_\rightarrow}$ 
 is a tree of finite increasing sequences of $\dom(f^{p_\rightarrow})$-objects such that for each $\vec \mu \in T$ the set 
 $Succ_T(\vec \mu) = \{ \nu : {\vec \mu}^\frown \langle \nu \rangle \in T \}$ is in the filter $E(\dom(f^{p_\rightarrow}))$.
\item A condition $p$ has the form  ${p_{\leftarrow}}^\frown \langle p_\rightarrow \rangle$,
  where each entry in $p_\leftarrow$ is a pair $(g, B)$ with $g$ a function and $B$ a tree, 
  where the pair $(g, B)$  is  defined (in essentially the same way that
 $p_\rightarrow$ was just defined from $\vec E$) from some extender sequence $\vec e \in V_\kappa$
  such that $\crit(\vec e) > \rho$ and $\vec e$ reflects the properties of $\vec E \restriction \xi$ in $\Ult(V, E_\xi)$ for some
  $\xi$ with $0 < \xi < \rho$. If $\nu$ is a $d$-object with $o(\nu) > 0$ for some $d$ as above then
  $\langle \nu(\bar\kappa)_{1 + i} : i < o(\nu) \rangle$ would be a typical value for $\vec e$.
\item  If $(g, B)$ is an entry in $p_\leftarrow$, then the associated extender sequence $\vec e$ can be computed
  by inspecting $\dom(g)$, whose least element is the tagged extender sequence
  $\langle \bar \kappa \rangle^\frown  \vec e$ where $\bar \kappa = \crit(\vec e)$.
\item   An entry $q$ in $p_\leftarrow$
  corresponding to $\vec e$ as above is a pair $(f^q, A^q)$
  where $\dom(f^q) \subseteq \{ \langle \beta \rangle^\frown \vec e : \crit(\vec e) \le \beta < h(\crit(\vec e)) \}$,
  the values of $f^q$ are finite increasing sequences of tagged extender sequences with non-increasing orders
   each less than $o(\vec e)$, and $A^q$ is  a tree of 
   finite increasing sequences of $\dom(f^q)$-objects,
   which has large branching with respect to a filter $e(\dom f^q)$.
 \item If the $i^{\rm th}$ entry in $p_\leftarrow$ is defined from a extender sequence $\vec e_i$, then $\crit(e_i)$ increases
   with $i$. 
 \item  A condition can be extended by refining existing entries, or by using sequences from the ``$A$-parts'':
   the second operation typically interpolates new entries between the entry from whose $A$-part the sequence was drawn and
   its immediate predecessor. For the sake of simplicity
   we only describe how to refine $p_\rightarrow$ and how to extend it using a sequence of length one
   $\langle \nu \rangle \in A^{p_\rightarrow}$. 

   A refinement of $p_\rightarrow = (f^{p_\rightarrow}, A^{p_\rightarrow})$ is a pair $(g, B)$
   where $\dom(f^{p_\rightarrow}) \subseteq  \dom(g)$, $g \restriction \dom(f^{p_\rightarrow}) = f^{p_\rightarrow}$,
   and $\{ \vec \nu \restriction \dom(f^{p_\rightarrow}) : \vec \nu \in B \} \subseteq A^{p_\rightarrow}$.

   If $\langle \nu \rangle \in A^{p_\rightarrow}$,
   and $f^{p_\rightarrow}(\bar \alpha)^\frown \langle \nu(\bar\alpha) \rangle$
   is increasing for all $\bar\alpha \in \dom(\nu)$, then we may extend by $\langle \nu \rangle$.
   
  In the special case of $o(\nu) = 0$ we just extend $f^{p_\rightarrow}(\bar \alpha)$
  to $f^{p_\rightarrow}(\bar \alpha)^\frown \langle \nu(\bar\alpha) \rangle$ for $\bar\alpha \in \dom(\nu)$
  and replace $A^{p_\rightarrow}$ by
  $A^{p_\rightarrow}_{\langle \nu \rangle} = \{  \vec \mu : \langle \nu \rangle^\frown \vec \mu \in A^{p_\rightarrow} \}$:
  no new entry is interpolated.

  When $o(\nu) > 0$ we write $f^p(\bar\alpha)$ as $x(\bar\alpha)^\frown y(\bar\alpha)$
  where $y(\bar\alpha)$ is the longest end-segment consisting of tagged extender sequences with order less than $o(\nu)$:
  we replace $f^{p_\rightarrow}(\bar\alpha)$ by $x(\bar\alpha)^\frown \langle \nu(\bar\alpha) \rangle$ for $\bar\alpha \in \dom(\nu)$
  and again replace $A^{p_\rightarrow}$ by $A^{p_\rightarrow}_{\langle \nu \rangle}$.  In this case we interpolate a new
  entry $(h, C)$ associated with the extender sequence
  $\vec e = \langle \nu(\bar\kappa)_{1 + i} : i < o(\nu) \rangle$:
  $\dom(h) = \rge(\nu)$,  $h(\nu(\bar\alpha)) = y(\bar\alpha)$ for each $\bar\alpha \in \dom(\nu)$,
  and $C = A^{p_\rightarrow} \downarrow \nu$ where $A^{p_\rightarrow} \downarrow \nu =
  \{  \vec \mu \circ \nu^{-1} : \mbox{$\vec \mu \in A^{p_\rightarrow}$ and for all $i$ $o(\nu_i) < o(\mu)$ and $\nu_i < \mu$} \}$.
\end{itemize}

We work below the condition with a single entry $(f, A)$ where $\dom(f) = \{ \bar \kappa \}$,
 $f(\bar\kappa) = \langle \rangle$,
and $A$ is the $f$-tree of all finite increasing sequences of $\dom(f)$-objects $\mu$ such that $o(\mu) < \rho$. 
As the definition of extension suggests,  for each $\bar\alpha \in \dom(f^{p_\rightarrow})$ a condition
contains finitely much information about an increasing $\rho$-sequence of tagged extender sequences: some of this
information is contained in $f^{p_\rightarrow}(\bar\alpha)$, but in general
$f^{p_\rightarrow}(\bar\alpha)$ also contains ``pointers'' (in the form of tagged extender sequences)
to extender sequences appearing in the entries of $p_\leftarrow$ and coordinates in those entries where more information
about the sequence associated with $\bar\alpha$ is to be found. 

More formally, let $G$ be $\mathbb P$-generic and work in $V[G]$. 
For each $\alpha$ with $\kappa \le \alpha < \lambda$ the generic sequence $G_\alpha$ is defined
to contain the tagged extender sequences which appear in $f^{p_\rightarrow}(\bar \alpha)$ for some $p \in G$,
enumerated in increasing order. For $j < \rho$ let $G_\alpha(j)$ be the $j^{\rm th}$ entry in $G_\alpha$,
and let $g_\alpha(j) = G_\alpha(j)_0$.

\begin{remark} \label{commute} 
In the light of the discussion above, it may seems counterintuitive that
the definition of $G_\alpha$ and $g_\alpha$ only uses $p_\rightarrow$. To clarify this point
consider how we may extend the trivial condition $(f, A)$ above to control the value of the first entry
$G_\alpha(0)$ in $G_\alpha$.
  We may use an object $\mu$ with $o(\mu) > 0$ and $\bar\alpha \in \dom(\mu)$ to extend
to a condition $p'$ with two entries in which $f^{p'_\rightarrow}(\bar\alpha) = \mu(\bar\alpha)$,
and then in the first entry of $p'$ use an object $\nu'$ of order zero with $\mu(\bar\alpha) \in \dom(\nu')$
to obtain a condition $p''$. Since $\nu' =  \nu \circ \mu^{-1}$ for some $\nu$ of order
zero,  we may also use $\nu$ first to extend to $q = \langle q_\rightarrow \rangle$ with
$f^{q_\rightarrow}(\bar\alpha) = \nu(\bar\alpha)$, and then produce $p''$ by using $\mu$. 
\end{remark}

The forcing poset $\mathbb P$ satisfies a version of the Prikry property, which we will state
formally in Section \ref{scales}. Roughly speaking, for any $p$ any question about the forcing extension can
be decided by refining the entries in $p$. It is also useful to note that if $p$ is a condition
with $p_\leftarrow$ nonempty, then below $p$ the forcing factors as 
${\mathbb P}/p \simeq {\mathbb P}'/p_{\leftarrow} \times {\mathbb P}/p_\rightarrow$,
 where the last entry in $p_\leftarrow$ is defined from an extender sequence $\vec e$,
and
${\mathbb P}'$ is defined from $\vec e$ and $h(\crit(\vec e))$ in the same way that
$\mathbb P$ is defined from $\vec E$ and $\lambda$. Note that $\vert {\mathbb P}' \vert = h(\crit(\vec e))$,
and that using the Prikry property for ${\mathbb P}/p_\rightarrow$ one can show that
$p$ forces  ``if $\crit(\vec e) = g_\kappa(j)$ then all subsets  
  of $g_\kappa(j + \omega)$ lie in the sub-extension by ${\mathbb P}'/p_{\leftarrow}$''. 

  Using the Prikry property and the factorisation, standard arguments show:
\begin{itemize}

\item $\mathbb P$ adds no new subsets of $\rho$, in particular $\rho$ is still regular and uncountable after forcing with
  $\mathbb P$. 
  
\item $\mathbb P$ preserves cardinals, and in the extension $\kappa$ is a strong limit cardinal of cofinality $\rho$.
  In particular $g_\kappa$ is a continuous, increasing and cofinal $\rho$-sequence in $\kappa$. 
  
\item  All the generic sequences $G_\alpha$ have order type $\rho$. 
  
\item If $\kappa < \gamma < \lambda$ then  $g_\kappa(i) < g_\gamma(i) < h(g_\kappa(i)) < g_\kappa(i+1)$ for all large $i$. 

\item $2^{g_\kappa(i)} = h(g_\kappa(i))$ for limit $i$ with $i < \rho$. 

\item  Let $\kappa < \gamma < \lambda$. In the generic extension, for all large $i$:
\begin{itemize}  
\item  If $\gamma$ is regular in $V$, then $g_\gamma(i+1)$ is regular in the generic extension.
\item  If $\gamma$ is measurable in $V$, then $g_\gamma(i+1)$ is measurable in the generic extension.
\item  If $\kappa < \cf(\gamma)$ in $V$, then $g_\kappa(i+1) < \cf(g_\gamma(i+1))$  in the generic extension.
\end{itemize}  

\item GCH holds in the intervals $[h(g_\kappa(i)), g_{\kappa}(i+\omega) )$ for $i$ limit, in particular
 if $\kappa < \gamma < \lambda$ then  GCH holds at $g_\gamma(i+1)$ for all large $i$. 
  
\end{itemize} 

\begin{remark} \label{changes}
  The forcing $\mathbb P$ is ${\mathbb P}_{\vec E, \lambda}$ from \cite{Carmi2011} with the following
  small changes and simplifications:
  \begin{itemize}
  \item Since $\lambda < j_{E_0}(\kappa)$, every coordinate $\bar\alpha$ consists of $\alpha$ followed
    by the whole extender sequence $\vec E$. 
  \item Because of the previous remark and the fact that $\rho < \kappa$,
 every tagged extender sequence
    in the range of a $d$-object contains the same extender sequence, and $E_\sigma(d)$ concentrates
    on objects of order $\sigma$.
  \item The definition of the ordering on $d$-objects is slightly more stringent than in
    \cite{Carmi2011}, but this is harmless because every $d$-object still has $E(d)$-many
    $d$-objects above it.
  \item The definition of the forcing guarantees that no new subsets of $\rho$ are added.
  \item If $q$ is an entry in $p_\leftarrow$ associated with $\vec e$, then the domain of
  $f^q$ can only contain sequences $\langle \beta \rangle^\frown \vec e$ for
  $\crit(\vec e) \le \beta < h(\crit(\vec e))$: this gives us better control over
  the continuum function in the generic extension, in particular it is why 
  $2^{g_\kappa(i)} = h(g_\kappa(i))$ for limit $i < \rho$. 
  \end{itemize}
\end{remark}

\section{Fat trees and iterations} \label{itns} 

Merimovich \cite{Carmi2007} used iterated ultrapowers to analyse names in the extension
by a ``one-extender'' extender based Prikry forcing. Roughly speaking, the iteration maps
afford a compact way of doing integration with respect to product measures which characterise
the trees appearing in the forcing conditions. We will carry out a similar construction here in the
more complicated context of our forcing poset $\mathbb P$ from Section \ref{EBF}: the situation here
is more complicated because in the context of \cite{Carmi2007} there is only one extender to iterate,
while here at stage $n$ we choose $\varepsilon < \rho$ and then apply $j_{0 n}(E_\varepsilon)$.

Recall that if $p$ is a condition with $p_\rightarrow = (f^{p_\rightarrow}, A^{p_\rightarrow})$ and
$\dom(f^{p_\rightarrow}) = d$, then $A^{p_\rightarrow}$ is a tree of finite increasing sequences
with $E(d)$-large branching at each node. Following Merimovich, we call such trees {\em $d$-trees,}
and introduce the related notion of a {\em $d$-fat tree}.

A {\em $d$-fat tree} is a tree $T$ of finite height consisting of finite increasing sequences of $d$-objects,
all of the same length, such that
for every non-maximal node $\vec \mu \in T$ there is $\varepsilon$ such that
$\{ \nu : \vec \mu^{\hskip2pt\frown} \langle \nu \rangle \in T \} \in E_\varepsilon(d)$.
Note that if a tree is $d$-fat then its intersection with any $d$-tree is also $d$-fat, in
particular it is non-empty. 

 Since $\rho < \kappa$, it is easy to see that any $d$-fat tree can be thinned to a $d$-fat subtree such that
  for every non-maximal level $l$, there is $\varepsilon_l < \rho$ such that all points
  on level $l$ have an $E_{\varepsilon_l}(d)$-large set of successors.  Thinning further
  we may also assume that a $d$-fat tree consists of sequences $\vec \nu$ such that
  $\nu_j$ has order $\varepsilon_j$ for all $j$; in a mild abuse of notation we
   say that $\vec \nu$ has order $\vec \varepsilon$.  We say that such a tree is {\em $(\vec \varepsilon, d)$-fat}.

  Given $\vec \varepsilon$ and $d$, define a finite iteration $j_{\vec \varepsilon}$ where we use the extender
  $j_{0i}(E_{\varepsilon_i})$ at stage $i$. As usual, for $m < n $ we let $j_{mn}$ denote the embedding
  from the $m^{\rm th}$ iterate to the $n^{\rm th}$ iterate. 

  \begin{lemma} \label{mc} 
    Let $
  mc_{\vec \varepsilon}(d) = \langle (j_{i \lh(\vec \varepsilon)} \restriction j_{0i}(d))^{-1} : i < \lh(\vec \varepsilon) \rangle.
$
Then $mc_{\vec \varepsilon}(d)$ is a maximal element in $j_{\vec \varepsilon}(T)$ for every $(\vec \varepsilon, d)$-fat tree $T$.
\end{lemma}

\begin{proof} 

  We prove this by induction on $n > 0$ where  $n=\lh(\vec\varepsilon)$.
  For use in the successor step we note that for every $\zeta < \rho$, since $\lambda < j_{E_\zeta}(\kappa)$
  we have that $d$ is fixed by $j_{E_\zeta}$: appealing to elementarity $j_{0i}(d)$ is fixed by
  $j_{i i+1}$ for all $i$. 
  
  \begin{itemize}

  \item Base case ($n=1$): The empty sequence has an $E_{\varepsilon_0}(d)$-large set of successors in $T$, 
    so by definition $mc_{\vec\varepsilon}(d) = \langle (j_{01} \restriction d)^{-1} \rangle$ is on level one of $j_{01}(T)$.

  \item Successor step: Suppose that $\vec\varepsilon$ has length $n+1$ and $T$ is
    $(\vec\varepsilon, d)$-fat. By the induction hypothesis,
    $mc_{\vec \varepsilon \restriction n}(d) \in j_{0 n}(T)$, where we have
    $mc_{\vec \varepsilon \restriction n}(d) = \discretionary{}{}{}
    \langle (j_{i n} \restriction j_{0i}(d))^{-1} : i < n \rangle$.

   Each entry in $mc_{\vec \varepsilon \restriction n}(d)$
    is a bijective partial function of size at most $j_{0 n-1}(\kappa)$,
    so that $j_{n n+1}$ maps it to its pointwise image: that is,
\[
j_{n n+1}(mc_{\vec \varepsilon \restriction n}(d)) = \langle (j_{i n+1} \restriction j_{0i}(d))^{-1} : i < n \rangle
\]

The tree $j_{0n}(T)$ is $(\vec\varepsilon, j_{0n}(d))$-fat, in particular
since  $mc_{\vec \varepsilon \restriction n}(d)$ is in $j_{0 n}(T)$ it
has a $j_{0n}(E_{\varepsilon_n}(d))$-large set of successors. 
The measure $j_{0n}(E_{\varepsilon_n}(d))$ is generated by the embedding
$j_{n n+1}$ together with the object  
$j_{0 n}(mc_{\varepsilon_n}(d)) = (j_{n n+1} \restriction j_{0n}(d))^{-1}$,
so that
\[
   j_{n n+1}(mc_{\vec \varepsilon \restriction n}(d))^\frown \langle j_{0 n}(mc_{\varepsilon_n}(d)) \rangle 
   = mc_{\vec\varepsilon}(d) \in j_{\vec\varepsilon}(T).
\]

  \end{itemize}

  This concludes the proof of Lemma \ref{mc}. 

  \end{proof}

The embedding $j_{\vec \varepsilon}$  and the sequence
$mc_{\vec \varepsilon}(d)$ can be used to characterise the $(\vec \varepsilon, d)$-fat trees:
for any tree of sequences $U$,
if $mc_{\vec\varepsilon}(d)$ is an element of $j_{\vec \varepsilon}(U)$
then $U$ contains a $(\vec \varepsilon, d)$-fat tree.

For use in Section \ref{scales}, we calculate some values of the entries in  $mc_{\vec \varepsilon}(d)$. 

\begin{lemma} \label{mcvaluelemma} 
  Let $mc = mc_{\vec \varepsilon}(d)$.
  For every $i < \lh(\vec\varepsilon)$ and $\bar \alpha \in d$,
  $mc_i(j_{\vec\varepsilon}(\bar\alpha)) = j_{\vec \varepsilon \restriction i}(\bar\alpha)$.
\end{lemma}

\begin{proof} 
  By definition $mc_i = (j_{i \lh(\vec \varepsilon)} \restriction j_{0i}(d))^{-1}$, and clearly
  $j_{0i} = j_{\vec\varepsilon \restriction i}$ and $j_{\vec\varepsilon} = j_{0 \lh(\vec\varepsilon)}$.
  Since $\bar\alpha \in d$,
  $j_{0i}(\bar\alpha) \in j_{0i}(d)$, and by the commutativity of the embeddings in the 
  iteration, $j_{i \lh(\vec\varepsilon)}(j_{0 i}(\bar\alpha)) = j_{\vec\varepsilon}(\bar\alpha)$.
  It follows that $j_{\vec\varepsilon}(\bar\alpha) \in \dom(mc_i)$ and
  $mc_i(j_{\vec\varepsilon}(\bar\alpha)) = j_{\vec \varepsilon \restriction i}(\bar\alpha)$.
\end{proof}

\section{Scale analysis} \label{scales} 

To use Lemma \ref{scalegeneration}, we need appropriate scales in the generic extension by
$\mathbb P$. For $\kappa \le \alpha < \lambda$ and $\eta < \rho$, let $g^*_\alpha(\eta) = g_\alpha(\omega^\eta + 1)$ 
 The scales we use will be appropriate initial segments of $\langle g^*_\alpha: \kappa \le \alpha < \lambda \rangle$. 

The choice of the indices $\omega^\eta+1$ may seem arbitrary, so we digress briefly to explain it.
Successor indices are needed because for limit $i$ the values of $g_\alpha(i)$ lie in an interval
 where the forcing has destroyed GCH, and this is bad for our intended application. 
 Indecomposable ordinals $\omega^\eta$ are useful because it is comparatively easy to
 design a condition which decides the values of $g_\alpha(\omega^\eta)$ and $g_\alpha(\omega^\eta +1)$,
  see for example the proof of Lemma \ref{strictincreasing} below. See also the discussion of the
 ``offset problem'' below. 

  Our analysis here owes an intellectual debt to work of Merimovich \cite{Carmi2007},
  who proved parallel results
  in the context of the generic $\omega$-sequences added by a ``one-extender'' Prikry forcing
  of the sort discussed in Section \ref{one-extender}. The analysis is harder in some
  respects and easier in others. 

  On the one hand, the complexity of the forcing $\mathbb P$ makes the analysis harder.
  To note a few salient points:
\begin{itemize}
\item The conditions themselves are more complex objects, in particular typically many entries
  in $p_{\leftarrow}$ will themselves contain extender sequences, functions and trees.
\item The connection between a condition $p$ and what it forces about the values
    $g_\alpha(i)$ of the generic functions is much more complex. 
\item All the objects in the one-extender forcing of Section \ref{one-extender} have order $0$ and
   behave in a rather uniform way, while in $\mathbb P$ objects of order $0$ and
   of positive order behave very differently.
 \item On a more technical note, we will need an analysis of
  dense open sets in $\mathbb P$.
 In the case of the one-extender forcing of Section \ref{one-extender} the parallel fact 
 just asserts that if $D$ is dense open and $(f, A)$ is a condition, then there exist an extension
 $(f', A')$ and an integer $n$ such that $f' \restriction \dom(f) = f$, and for every $\vec \nu'$
  of length $n$ in $A'$ the minimal extension of $(f', A')$ using $\nu'$ lies in $D$. 
 Compare this with Lemmas \ref{410} and \ref{411} below. 
\end{itemize}

  On the other hand, one source of difficulty in the one-extender case is that the
  $\omega$-sequences assigned to different coordinates  are ``offset'' from each other
  by a finite amount.  To see the issue note that
 if $f^p(\alpha)$ and $f^p(\beta)$ are finite increasing sequences of ordinals
  with different lengths, and $A^p$ is a tree consisting of objects which all have $\alpha$ and $\beta$
  in their domain, then $\nu(\alpha) < \nu(\beta)$ for all $\nu$ appearing in $A^p$ but
  $p$ only forces that a shifted version of the sequence at $\alpha$ is eventually dominated
  by the sequence at $\beta$.  In our case we are able to avoid this difficulty: the point is that
  since $\rho$ is a regular uncountable cardinal it is a limit of indecomposable ordinals,
  that is those of form $\omega^\eta$, and this helps us argue
  (see Lemma \ref{strictincreasing}) that the sequences $g^*_\alpha$ 
  and $g^*_\beta$ will eventually ``synchronise'' and no offset is needed.    

  If $(f, A)$ is an entry in a condition and $\vec \nu \in A$ then we write $(f, A)_{\vec \nu}$ for
  the sequence of entries obtained by extending $(f, A)$ using each entry in $\vec \nu$ in turn.
  If $p = \langle p_\rightarrow \rangle$ and $\vec \nu \in A^{p_\rightarrow}$ then
  we write $p_{\vec \nu}$ for the condition $(p_\rightarrow)_{\vec \nu}$. In the more general situation
  where $p_\leftarrow$ is non-empty and $\vec \nu \in A^{p_\rightarrow}$, we let
  $p_{\vec \nu} = {p_{\leftarrow}}^\frown (p_{\rightarrow})_{\vec \nu}$. 

\begin{lemma} \label{strictincreasing} 
 $\langle g^*_\alpha : \kappa \le \alpha < \lambda \rangle$
 is strictly increasing in the eventual domination ordering.
\end{lemma}

\begin{proof}
  Let $p$ be a condition and let $\alpha < \beta < \lambda$.
  Let $\eta < \rho$ be so large that $o(\vec e) < \eta$ for every
   extender sequence $\vec e$ associated with an entry in $p_{\leftarrow}$, and
  also $o(x) < \eta$ for every tagged extender sequence $x$ appearing in
  $f^{p_\rightarrow}(\alpha)$ or $f^{p_\rightarrow}(\beta)$.

  Choose an object $\mu$ such that $\langle \mu \rangle \in A^{p_\rightarrow}$,
  $o(\mu) = \eta$, $\mu(\bar\kappa)_0 > \crit(\vec e)$ for every extender sequence $\vec e$ associated with an entry
  in $p_{\leftarrow}$, and $\bar\alpha, \bar\beta \in \dom(\mu)$. Form the condition $p_{\langle \mu \rangle}$,
  and refine the $A$-parts of the entries in $p_{\langle \mu \rangle \leftarrow}$ so that
  only objects of order less than $\eta$ appear, to obtain a condition $q$.

  Note that $f^{q_\rightarrow}(\bar\alpha) = f^{q_\rightarrow}(\bar\beta) = \langle\rangle$,
  and also $\bar\alpha, \bar\beta \in \dom(\nu)$ for all $\nu$ appearing in $A^{q_\rightarrow}$. 
  It is routine to check that $q$ forces that:
  \begin{itemize}
  \item $g_\alpha(\omega^\eta) = \mu(\bar\alpha)$ and $g_\beta(\omega^\eta) = \mu(\bar\beta)$. 
  \item For all $i$ with $\omega^\eta < i < \rho$, there is $\nu$ appearing in  $A^{q_\rightarrow}$
    such that $g_\alpha(i) = \nu(\bar\alpha)$ and $g_\beta(i) = \nu(\bar\beta)$.
   In particular   $g_\alpha(i) < g_\beta(i)$  since $\nu$ is order-preserving.
 \item  $g^*_\alpha <^* g^*_\beta$.
 \end{itemize} 

\end{proof} 

 Lemma \ref{scale-analysis} {below} is our main technical result. In its proof we will use
two lemmas from \cite{Carmi2011} which give an analysis of dense sets
in the forcing. For the reader's convenience we quote those lemmas here.

We start with the definitions of {\em Prikry extension} and {\em strong Prikry extension} \cite[Definition 4.5]{Carmi2011}. 
\begin{itemize}
\item If
  $p$ and $q$ are conditions in $\mathbb P$ with
  $p = \langle p_\rightarrow \rangle$ and $q = \langle q_\rightarrow \rangle$, 
  then $p$ is a {\em Prikry extension of $q$} ($p \le^* q$)
  if and only if $f^{p_\rightarrow}  \restriction \dom(f^{q_\rightarrow}) = f^{q_\rightarrow}$
  and $A^{p_\rightarrow} \restriction \dom(f^{q_\rightarrow}) \subseteq A^{q_\rightarrow}$.
More generally for $p, q \in {\mathbb P}$ we define recursively $p \le^* q$
if and only if $p_\rightarrow \le^* q_{\rightarrow}$ and $p_\leftarrow \le^* q_{\leftarrow}$: unwrapping the recursion,
this implies that $p$ and $q$ contain the same number of entries and each entry in $p$ is a Prikry extension
of the corresponding entry in $q$. 
\item  For   $p$ and $q$  conditions in $\mathbb P$ with
  $p = \langle p_\rightarrow \rangle$ and $q = \langle q_\rightarrow \rangle$, 
$p$ is a {\em strong Prikry extension of $q$} ($p \le^{**} q$)
if and only if $f^{p_\rightarrow} = f^{q_\rightarrow}$ and $A^{p_\rightarrow} \subseteq A^{q_\rightarrow}$.
 The extension to arbitrary $p$ and $q$ is defined as for the notion of Prikry extension.
\end{itemize} 

The precise statement of the Prikry lemma is that for every condition $q$ 
and every sentence $\phi$ of the forcing language, there is $p \le^* q$ such that
$p$ decides $\phi$. The following fact  is crucial in the discussion that follows:

\begin{fact} \label{orderfact}
  If $p = \langle p_\rightarrow \rangle$ and $q \le p$, then there is a unique $\vec \nu \in A^{p_\rightarrow}$
  such that $q \le^* p_{\vec \nu}$.
\end{fact}

We refer the reader to \cite[Definition 4.5]{Carmi2011} and  the discussion
in Remark \ref{commute} for more on the ordering of conditions in extender-based Radin forcing.  
In the situation of Fact \ref{orderfact}, or the more general one where $p_\leftarrow$ is non-empty and it's
only a tail of $q$ that is a Prikry extension of $(p_\rightarrow)_\nu$,
we say that {\em $q$ is an extension of $p$ by $\vec \nu$.}

We can view the construction of $q$ from $p$ as happening in stages: first we form $p_{\vec \nu}$ (the minimal extension
of $p$ by $\vec \nu$), then we take a Prikry extension of each entry in $p_{\vec \nu}$. Taking an even more granular
approach, at each entry we can view the process of Prikry extension as occurring by first extending the $f$-part,
and then forming a strong Prikry extension of the resulting entry by shrinking the $A$-part. 
 
Before stating Lemmas \ref{410} and \ref{411}, we need one more piece of notation:
if $T$ is a $d$-fat tree,  $r$ is a function with domain $T$ and $\vec \nu \in T$ then
${\vec r}(\vec \nu) = \langle r(\vec \nu \restriction i) : 0 < i \le \lh(\vec \nu) \rangle$.

The following facts appear as Lemmas 4.10 and 4.11 in \cite{Carmi2011}.

\begin{lemma} \label{410}
    Let $p \in {\mathbb P}$, let $T \subseteq A^{p_\rightarrow}$ be a $\dom(p_{\rightarrow})$-fat subtree, and
    let $r$ be a function with domain $T$ such that
    ${\vec r}(\vec \nu) \le^{**} p_{\vec \nu \leftarrow}$ for every maximal $\vec \nu \in T$. Then there is a
    strong Prikry extension $q \le^{**} p$ such that $q_{\leftarrow} = p_{\leftarrow}$ (that is $q$ is obtained by
    merely replacing $p_{\rightarrow}$ by some strong Prikry extension $q_\rightarrow$) and
    the set of conditions of form ${p_{\leftarrow}}^\frown {\vec r}(\vec \nu)^\frown \langle p_{\vec \nu \rightarrow} \rangle$
    for $\vec \nu \in T$ maximal is predense below $q$. 
\end{lemma}

\begin{lemma} \label{411}
    Let $p \in {\mathbb P}$ with $p = \langle p_\rightarrow \rangle$, and let $D \subseteq {\mathbb P}$
    be a dense open set. Then there exist a Prikry extension $q \le^* p$, a $\dom(f^{q_\rightarrow})$-fat tree
    $T \subseteq A^{q_\rightarrow}$ and a function $r$ with domain $T$ such that
    ${\vec r}(\vec \nu) \le^{**} q_{\vec \nu \leftarrow}$
    and  ${\vec r}(\vec \nu)^\frown \langle q_{\vec \nu \rightarrow} \rangle \in D$  for every maximal $\vec \nu \in T$.
\end{lemma} 

Of course the dense sets which we need to analyse are rather special,
but we have chosen to use the general machinery of \cite{Carmi2011} rather than
reprove the relevant special cases.

\begin{lemma} \label{scale-analysis} 
If $\kappa < \gamma < \lambda$ with $\cf(\gamma) > \kappa$, then
$g^*_\gamma$ is an exact upper bound for $\langle g^*_\delta : \kappa \le \delta < \gamma \rangle$.
\end{lemma} 

\begin{proof}

 Let $p \in {\mathbb P}$ and $\langle \dot \tau_\eta : \eta < \rho \rangle$ be 
 such that $p \forces \forall \eta < \rho \; \dot \tau_\eta < \dot g^*_\gamma(\eta)$.
 We will ultimately produce a condition $q \le p$, together with ordinals $\delta < \gamma$ and $\eta < \rho$ such that
 $q \forces \forall \zeta > \eta \; \tau_\zeta < \dot g^*_\delta(\zeta)$.

\begin{claim}
  Extending $p$ if necessary, we may assume that: 
  \begin{itemize}
  \item Both $\bar\kappa$ and $\bar\gamma$ are in $\dom(f^{p_\rightarrow})$. 
  \item For every object $\nu$ appearing in any sequence from  $A^{p_\rightarrow}$, 
  both $\bar \kappa$ and $\bar \gamma$ appear in $\dom(\nu)$, and 
   $\cf(\nu(\bar \gamma)_0) > \nu(\bar\kappa)_0$. 
  \item  $f^{p_\rightarrow}(\bar \kappa) = f^{p_\rightarrow}(\bar \gamma) = \langle \rangle$.
  \item There is $\eta < \rho$ such that: 
\begin{itemize} 
    \item  $p_{\leftarrow}$ determines the values of $g_\kappa(\omega^\eta)$ and $g_\gamma(\omega^\eta)$,
    say as $\kappa^*$ and $\gamma^*$.
     \item If $q \le p$ via some $\vec \nu \in A^{p_\rightarrow}$, and $q$ determines
      $g_\kappa(i)$ or $g_\gamma(i)$ for some $i$ with $\omega^\eta < i$,
     then the minimal extension $p_{\vec \nu}$ already determines $g_\kappa(i)$ and $g_\mu(i)$.
      \item  $p$ forces that $\cf(g_\gamma(i)) > g_\kappa(i)$ for all successor $i > \omega^\eta$.  
\end{itemize} 
\end{itemize} 
\end{claim}

\begin{proof}
  Choose $\eta$ as in the proof of Lemma \ref{strictincreasing}, 
  and then replace $p$ by a suitable strong Prikry extension
   of $p_{\langle \mu \rangle}$ for some $\langle \mu \rangle \in A^{p_\rightarrow}$
  with $o(\mu) = \eta$.  Now $p_\leftarrow$ determines
  $g_\kappa(\omega^\eta)$ as $\mu(\bar\kappa)_0$
  and $g_\gamma(\omega^\eta)$ as $\mu(\bar\gamma)_0$. If $q$ extends $p$ via $\vec \nu$, then
  $\bar\kappa$ and $\bar\gamma$ are in $\dom(\nu_k)$ for all $k$, and the minimal extension already
  determines the relevant values. 
\end{proof}

\begin{claim} Without loss of generality, we may assume that $p = p_{\rightarrow}$.
\end{claim} 
 
\begin{proof} 
  As in the discussion at the end of Section \ref{EBF},
  ${\mathbb P}/p$ is isomorphic to the product of a  ``low part'' ${\mathbb P}_{\rm low}$
   below $p_\leftarrow$ and 
   a ``high part'' ${\mathbb P}_{\rm high}$  below $p_{\rightarrow}$.
   We view $\dot \tau_\eta$ as a ${\mathbb P}_{\rm high}$-name
  for a ${\mathbb P}_{\rm low}$-name. Since ${\mathbb P}_{\rm low}$ has size $h(\kappa^*)$,
  and $p$ forces $\cf(g^*_\gamma(\zeta)) > g^*_\kappa(\zeta)$ for all $\zeta > \eta$, 
  it is easy to find  ${\mathbb P}_{\rm high}$-names $\dot \sigma_\zeta$ for $\zeta > \eta$ such that
  $p \forces \forall \zeta > \eta \; \dot \tau_\zeta < \dot \sigma_\zeta < g^*_\gamma(\zeta)$.
  Replacing  $\mathbb P$ by ${\mathbb P}_{\rm high}$, $p$ by $p_{\rightarrow}$
  and $\dot \tau_\zeta$ by {$\dot \sigma_{\zeta}$}, we have the claim. 
\end{proof}

In the light of the preceding Claim, it is clearly sufficient to prove that we can find $\delta < \gamma$
such that $p$ extends to
a condition forcing $\forall \zeta < \rho \; \dot\tau_\zeta < g_\delta^*(\zeta)$. 

  For each $\zeta$ with $\zeta < \rho$, let $D_\zeta$ be the dense open set of conditions
  $t$ in ${\mathbb P}$ such that:
  \begin{itemize}
  \item $t$ determines the values of $\tau_\zeta$ and $g^*_\gamma(\zeta)$.
  \item $t_\leftarrow$ has at least one entry defined from an extender sequence with order
  $\zeta$.\footnote{{$\hskip2pt$} The first requirement on $t$ actually implies the second one, but we preferred to make this
    point explicit.}
   \end{itemize}

\begin{claim} \label{letscallthis} 
  There exist $q \le^* p$, integers $n_\zeta$, $\dom(f^{q_\rightarrow})$-fat subtrees $T_\zeta$ of $A^{q_\rightarrow}$  
  with height $n_\zeta$, and functions $R_\zeta$ and $h_\zeta$ for $\zeta < \rho$ with the following properties.
  For all $\zeta < \rho$ and all maximal $\vec \nu \in T_\zeta$:
  \begin{itemize}
  \item ${\vec R}_\zeta(\vec \nu) \le^{**} q_{{\vec \nu}\leftarrow}$.
  \item ${\vec R}_\zeta(\vec \nu)^\frown \langle q_{{\vec \nu}\rightarrow} \rangle \in D_\zeta$.
  \item
  $h_\zeta(\vec \nu)$ is the value which ${\vec R}_\zeta(\vec \nu)^\frown \langle q_{{\vec \nu}\rightarrow} \rangle$
  determines for $\tau_\zeta$. 
  \end{itemize}
\end{claim}

\begin{proof} 

  We will build a $\le^*$ decreasing chain $\langle q^\zeta : \zeta < \rho \rangle$, together with
  trees $S^\zeta$ and functions $R^\zeta$, such that:
  \begin{itemize}
  \item $q^0 = p$.
  \item $S^\zeta$ is a $\dom(f^{q^{\zeta+1}_\rightarrow})$-fat tree.
  \item For all maximal $\vec \nu \in S^\zeta$,
   ${\vec R^\zeta}(\vec \nu) \le^{**} q^{\zeta+1}_{{\vec \nu}\leftarrow}$
    and ${\vec R}^\zeta(\vec \nu)^\frown \langle q^{\zeta+1}_{{\vec \nu}\rightarrow} \rangle \in D_\zeta$.
  \end{itemize}
  Once we have chosen $q^\zeta$, we appeal to Lemma \ref{411} to produce $q^{\zeta+1} \le^* q^\zeta$
  together with $S^\zeta$ and $R^\zeta$. To choose $q^\zeta$ for $\zeta$ limit we use the
  $\kappa$-completeness of the $\le^*$ ordering and the fact that $\rho < \kappa$.

  At the end of the construction, let $q$ be a lower bound in the $\le^*$-ordering for
  the sequence $\langle q^\zeta: \zeta < \rho \rangle$. 

  \begin{subclaim} \label{thinkofone} 
    For every $\zeta < \rho$: 
\begin{itemize}
\item  For every $\vec \nu \in A^{q_\rightarrow}$,
       $q_{\vec \nu} \le^* q^\zeta_{\vec \nu \restriction \dom(f^{q^\zeta_\rightarrow})}$.
\item If we let $T_\zeta = \{ \vec \nu \in A^{q_\rightarrow} : \vec \nu \restriction \dom(f^{q^{\zeta+1}_\rightarrow})  \in S^\zeta \}$, then
    $T_\zeta$ is a $\dom(f^{q_\rightarrow})$-fat tree with the same height as $S^\zeta$. 
\item
    There exists a  function $R_\zeta$ with domain $T_\zeta$, such that   
    ${\vec R}_\zeta(\vec \nu) \le^{**} q_{{\vec \nu}\leftarrow}$ and
    ${\vec R}_\zeta(\vec \nu) \le^* {\vec R}^\zeta(\vec \nu \restriction \dom(f^{q^{\zeta+1}_\rightarrow}))$.
\end{itemize}
\end{subclaim} 

\begin{proof} We take each assertion in turn.

\begin{itemize}

\item  To lighten the notation, let  $f=f^{q_\rightarrow}$, $A = A^{q_\rightarrow}$, $q' = q^\zeta$, $f' = f^{{q^\zeta}_\rightarrow}$,
 $A' = A^{{q^\zeta}_\rightarrow}$,  and  $d' = \dom(f')$
 Let $\vec \nu \in A$, and note that since $q \le^* q'$
  we have $\nu' \in A'$  where $\nu' = \vec \nu \restriction d'$. Note also that
  $o(\nu_i) = o(\nu'_i)$ for all $i < \lh(\nu)$.

  Now we compare the construction process for entries in $q_{\vec \nu}$ and $q'_{\vec \nu'}$. 
  Since $q \le^* q'$, $f \restriction d' = f'$
  and $A \restriction d' \subseteq A'$.
  By definition $A^{q_{{\vec \nu} \rightarrow}} = A^{q_\rightarrow}_{\vec \nu}$: if $\vec \mu \in A^{q_{{\vec \nu} \rightarrow}}$ then
  ${\vec \nu}^{\hskip2pt\frown} \vec \mu \in A$, so 
  $({\vec \nu} \restriction d')^\frown (\vec \mu \restriction d') \in A'$,
  and hence  $\vec \mu \restriction d' \in A'_{\vec \nu'}  =  A^{q'_{\vec \nu' \rightarrow}}$. As for the $f$-parts,
  $\dom(f^{q_{{\vec \nu}\rightarrow}}) = \dom(f)$,  $\dom(f^{q'_{{\vec \nu'}\rightarrow}}) = \dom(f') = d'$, 
  and the value of $f^{q_{{\vec \nu}\rightarrow}}(\bar\alpha)$ depends only on $f^{q_\rightarrow}(\bar\alpha)$ and $\nu(\bar\alpha)$,
  so that $f^{q_{{\vec \nu}\rightarrow}} \restriction d'  = f^{q'_{{\vec \nu'}\rightarrow}}$. 

  The argument comparing entries in $q_{{\vec \nu}\leftarrow}$ and $q'_{{\vec \nu'}\leftarrow}$ is quite similar. Suppose that
  $o(\nu_i) > 0$, so that using $\nu_i$ generates an entry $(g, B)$ in $q_{{\vec \nu}\leftarrow}$ with
  $(g', B')$ the corresponding entry in $q'_{{\vec \nu'}\leftarrow}$. Note that $\dom(g) = \rge(\nu_i)$ and $\dom(g') = \rge(\nu'_i)$.
 The value of $g(\nu_i(\bar\alpha))$ depends only on the
  values of $\nu_j(\bar\alpha)$ (for $j$ such that $\bar\alpha \in \dom(\nu_j)$) and $f^{q_\rightarrow}(\bar\alpha)$,
  so that easily $g \restriction \dom(g') = g'$. If $\vec \mu \in B = A \downarrow \nu_i$,    
  then $\vec \mu = \vec \mu^* \circ \nu_i^{-1}$ for some $\vec \mu^* \in A$ with $o(\mu^*_k) < o(\nu_i)$ and  $\mu^*_k < \nu_i$
  for all $k < \lh(\vec \mu^*)$. Then $\vec \mu^* \restriction d' \in A'$,  $o(\mu^*_k \restriction d') < o(\nu_i')$
  and $\mu^*_k \restriction d' < \nu_i'$
  for all $k$, so that $\vec \mu \restriction \dom(g') \in B' = A' \downarrow \nu'$. 
  
\item Since $S^\zeta$ is a tree it is easy to see that $T_\zeta$ is a tree. 
 Let $\vec \nu \in T_\zeta$, let $d' = \dom(f^{q^{\zeta+1}_\rightarrow})$ and let $\vec \nu' = \nu \restriction d'$,
  so that $\vec \nu' \in S^\zeta$. Suppose that $\nu'$ is not maximal in $S^\zeta$.
 Since $S^\zeta$ is a fat tree there is an $i < \rho$ such that
 $\{ \mu' : {\vec \nu'}{}^\frown \langle \mu' \rangle \in S^\zeta \} \in E_i(d')$, and
 then since $A^{q_\rightarrow}$ is an $E(\dom(f^{q_\rightarrow})$-tree and $\vec \nu \in A^{q_\rightarrow}$ 
 we also have $\{ \mu : {\vec \nu}^\frown \langle \mu \rangle \in A^{q \rightarrow} \} \in E_i(\dom(f^{q_\rightarrow})$
 Since $E_i(d')$ is the projection of $E_i(\dom(f^{q_\rightarrow})$ via the restriction map
$\mu \mapsto \mu \restriction d'$, we also have 
$\{ \mu : {\vec \nu'}{}^\frown \langle \mu \restriction d' \rangle \in S^\zeta \} \in E_i(\dom(f^{q_\rightarrow})$.
 So $\{ \mu : {\vec \nu}^\frown \langle \mu \rangle \in T_\zeta \} \in E_i(\dom(f^{q_\rightarrow})$.

 It follows easily that $T_\zeta$ is a fat tree with the same height as $S^\zeta$. In particular 
 if $\vec \nu \in T_\zeta$ is maximal  then $\vec \nu \restriction d'$ is maximal in $S^\zeta$.  
  
\item
  To define $R_\zeta(\vec \nu)$ for $\vec \nu \in T^\zeta$, let
  $\nu' = \vec \nu \restriction \dom(f^{q^{\zeta+1}_\rightarrow})$ so that $\nu' \in S^\zeta$. 
  By the choice of $R^\zeta$,  ${\vec R^\zeta}(\vec \nu') \le^{**} q^{\zeta+1}_{{\vec \nu'}\leftarrow}$.  
  As we just showed, $q_{\vec \nu} \le^* q^{\zeta+1}_{\vec \nu'}$.

  Let $\nu_i$ be the last entry in $\vec \nu$. If $o(\nu_i) = 0$ there is nothing to do,
  so assume that  $o(\nu_i) > 0$. Let $(g', B')$ and $(g, B)$ be the last entries
  in $q^{\zeta+1}_{{\vec \nu'}\leftarrow}$ and $q_{{\vec \nu}\leftarrow}$ respectively,
  so that they correspond to $\nu_i'$ and $\nu_i$.  
  Let $(g', C') = {\vec R^\zeta}(\vec \nu')$, so that $C' \subseteq B'$.
  Let $C = \{ \vec \mu \in B : \vec \mu \restriction \dom(g') \in C' \}$
    and note that $(g, C)$ is a legitimate entry with $(g, C) \le^{**} (g, B)$ and $(g, C) \le^* (g', B')$.
  Set  $R_\zeta(\vec \nu)$ equal to $(g, C)$. 

\end{itemize} 

   This concludes the proof of Subclaim \ref{thinkofone}.
\end{proof}

   Let $\vec \nu \in T_\zeta$ be maximal.
   By Subclaim \ref{thinkofone},   ${\vec R}_\zeta(\vec \nu) \le^{**} q_{{\vec \nu}\leftarrow}$.   
   By the choice of $R^\zeta$, we have ${\vec R}^\zeta(\vec \nu \restriction \dom(f^{q^{\zeta+1}_\rightarrow}))^\frown \langle q^{\zeta+1}_{{\vec \nu}\rightarrow} \rangle \in D_\zeta$.
   By Subclaim \ref{thinkofone} again, ${\vec R}_\zeta(\vec \nu) \le^* {\vec R}^\zeta(\vec \nu \restriction \dom(f^{q^{\zeta+1}_\rightarrow}))$,
   and also $q_{\vec \nu} \le^* q^\zeta_{\vec \nu \restriction \dom(f^{q^\zeta_\rightarrow})}$.
   It follows that
   ${\vec R}_\zeta(\vec \nu)^\frown \langle q_{{\vec \nu}\rightarrow} \rangle \le^* {\vec R}^\zeta(\vec \nu \restriction \dom(f^{q^{\zeta+1}_\rightarrow}))^\frown \langle q^{\zeta+1}_{{\vec \nu}\rightarrow} \rangle$, 
    and so since $D_\zeta$ is open that ${\vec R}_\zeta(\vec \nu)^\frown \langle q_{{\vec \nu}\rightarrow} \rangle \in D_\zeta$.
   By the definition of $D_\zeta$, we may now choose $h_\zeta(\vec \nu)$ to be the value which ${\vec R}_\zeta(\vec \nu)^\frown \langle q_{{\vec \nu}\rightarrow} \rangle$
  determines for $\tau_\zeta$.

   This concludes the proof of Claim \ref{letscallthis}.
\end{proof} 

Replacing $T_\zeta$ by a subtree if necessary, we may assume that 
for every $\zeta < \rho$ there is a sequence 
${\vec \varepsilon}_\zeta$ such that 
 $T_\zeta$ is a $( {\vec \varepsilon}_\zeta, \dom(f^{q_\rightarrow}))$-fat tree.
It is immediate from the definition of the set $D_\zeta$ that $\zeta$ appears at
least once in the sequence ${\vec \varepsilon}_\zeta$.

Let $\zeta < \rho$  and let the first appearance of $\zeta$
in ${\vec \varepsilon}_\zeta$ have index $n_\zeta$.
  Shrinking the values of the function $R_\zeta$ if necessary, we may assume that
  for all maximal $\vec \nu \in T_\zeta$, all objects appearing in the tree
  parts of the entries in ${\vec R}_\zeta(\vec \nu \restriction n_\zeta)$ 
  have order less than $\zeta$. The advantage of this is that
  now for every maximal $\vec \nu \in T_\zeta$,
  ${\vec R}_\zeta(\vec \nu)^\frown \langle q_{{\vec \nu}\rightarrow} \rangle$ 
  decides the value of  $g_\gamma(\omega^\zeta)$ as $\nu_{n_\zeta}(\bar\gamma)_0$. 

  \begin{claim}
    Let $\zeta < \rho$, ${\vec \varepsilon} = {\vec \varepsilon}_\zeta$ and $n = n_\zeta$. Then:
\begin{itemize}
\item $\varepsilon_{n + 1} = 0$. 
\item For all maximal $\vec \nu \in T_\zeta$,
  the condition    ${\vec R}_\zeta(\vec \nu)^\frown \langle q_{{\vec \nu}\rightarrow} \rangle$
  decides the value of $g^*_\gamma(\zeta)$ as $\nu_{n+1}(\bar \gamma)_0$. 
\end{itemize}
\end{claim}

\begin{proof}
   For the first claim, suppose for a contradiction that either $\varepsilon_n$ is the last entry of
  ${\vec \varepsilon}$ or $\varepsilon_{n+1} > 0$. Let $\vec \nu \in T_\zeta$ be maximal
  and let $t = {\vec R}_\zeta(\vec \nu)^\frown \langle q_{{\vec \nu}\rightarrow} \rangle$ 
  so that $t$ determines the value of $g_\gamma(\omega^\zeta)$ as $\nu_n(\bar\gamma)_0$,
 and the value $g_\gamma^*(\zeta)$ as $\theta$ say.  

  At this point we need to be slightly careful, and keep in mind that when we use objects
  of order zero in $\vec \nu$ in the construction of $q_{\vec \nu}$ they do not give rise to new entries.
  Accordingly let $\varepsilon_n$ have index $m$ in the increasing enumeration of the non-zero entries
  of $\vec\varepsilon$, and note that if it exists $\varepsilon_{n+1}$ has index $m+1$. 

   Since entry $m+1$ in $t$ is defined from an extender sequence $\vec e$ with
  $\crit(\vec e) > \theta$, we may now extend $t$ using an object of order zero drawn from the tree part
   of entry $m+1$ to force $g_\gamma(\omega^\zeta+1) > \theta$. This contradiction establishes the first claim,
   and the second claim follows immediately. 
\end{proof}

For each $\zeta < \rho$, we form the iteration $j_{{\vec \varepsilon}_\zeta}$ as in Section \ref{itns}.

\begin{claim} Let $\zeta < \rho$, ${\vec \varepsilon} = {\vec \varepsilon}_\zeta$, $n = n_\zeta$
  and $mc = mc_{\vec \varepsilon}(\dom f^q)$.
  Then for all large $\delta < \gamma$, 
$j_{\vec \varepsilon}(h_\zeta)(mc) < j_{\vec \varepsilon \restriction n +1}(\delta)$

\end{claim}  

\begin{proof}
  Recall that for $\vec \nu$ maximal in $T_\zeta$,
  $h_\zeta(\vec \nu)$ is the value which ${\vec R}_\zeta(\vec \nu)^\frown \langle q_{{\vec \nu}\rightarrow} \rangle$
  determines for $\tau_\zeta$, and  $\nu_{n+1}(\bar \gamma)_0$ is the value it determines for $g^*_\gamma(\zeta)$,
  so $h_\zeta(\vec \nu) < \nu_{n+1}(\bar \gamma)_0$.
  Recall also that $mc$ is a maximal element in $j_{\vec\varepsilon}(T_\zeta)$. 
  
By elementarity,
$j_{\vec \varepsilon}(h_\zeta)(mc) < mc_{n+1}(j_{\vec \varepsilon}(\bar \gamma))_0$.
By Lemma \ref{mcvaluelemma},
$mc_{n+1}(j_{\vec \varepsilon}(\bar \gamma)) = j_{\vec \varepsilon \restriction n +1}(\gamma)$.
Since $\cf(\gamma) > \kappa$, and $j_{\vec \varepsilon \restriction n +1}$ can be represented as the ultrapower by
 a short extender with critical point $\kappa$,
 we have that $j_{\vec \varepsilon \restriction n +1}$ is continuous at $\gamma$ and the claim follows.
\end{proof} 

Recall that $\dom(f^{q_\rightarrow})$ is bounded in $\gamma$, because $\cf(\gamma) > \kappa$. 
Hence we can choose $\delta < \gamma$ so large that $\bar\delta \notin \dom(f^{q_\rightarrow}) \cap \gamma$, and
$j_{{\vec \varepsilon}_\zeta}(h_\zeta)(mc_{{\vec \varepsilon}_\zeta}(\dom(f^{q_\rightarrow}))) < j_{{\vec \varepsilon}_\zeta \restriction n +1}(\delta)$
for all $\zeta < \rho$. 
We find $q' \le^* q$ such that $\bar\delta \in \dom(f^{q'_\rightarrow})$ and $f^{q'_\rightarrow}(\bar \delta) = \langle \rangle$.
For each $\zeta < \rho$ we  choose a $({\vec \varepsilon}_\zeta, \dom(f^{q'_\rightarrow}))$-fat
tree $T_\zeta'$ and a function $R'_\zeta$ so that
$T'_\zeta \restriction \dom(f^q) \subseteq T_\zeta$ and
$\vec{R'_\zeta}(\vec \mu) \le^* \vec {R_\zeta}(\vec \mu \restriction \dom(f^q))$
 for all maximal $\vec \mu \in T'_\zeta$.

 For all $\zeta < \rho$ we have
$mc_{{\vec \varepsilon}_\zeta}(\dom(f^{q'_\rightarrow})) \in j_{{\vec \varepsilon}_\zeta}(T_\zeta')$,
 and by the choice of $\delta$ and Lemma \ref{mcvaluelemma}   
 \[
 j_{{\vec \varepsilon}_\zeta}(h_\zeta)(mc_{{\vec\varepsilon}_\zeta}(\dom(f^{q'}))) <
 mc_{{\vec \varepsilon}_\zeta}(\dom(f^{q'_\rightarrow}))_{n+1}(j_{{\vec \varepsilon}_\zeta}(\bar\delta)).
 \]

Using the connection between fat trees and iterations, and elementarity,
we choose fat subtrees $T''_\zeta \subseteq T'_\zeta$ such that for all maximal $\vec \nu \in T''_\zeta$,
$\vec{R'_\zeta}(\vec \nu)^\frown \langle q'_{{\vec \nu}\rightarrow} \rangle$ forces 
$\tau_\zeta  < g^*_\delta(\zeta)$, where the key point is that 
$\vec{R'_\zeta}(\vec \nu)^\frown \langle q'_{{\vec \nu}\rightarrow} \rangle$
decides the value of $g^*_\delta(\zeta)$   as $\nu_{n+1}(\bar\delta)$. 

  Now we make $\rho$ many applications of Lemma \ref{410} to get $r \le^{**} q'$ such that for every
  $\zeta < \rho$, the set of conditions of form ${\vec R'}_\zeta({\vec \nu})^\frown \langle q'_{\vec \nu \rightarrow} \rangle$
  with $\vec \nu \in T''_\zeta$ maximal is predense below $r$. Then $r$ forces that $\tau_\zeta < g^*_{\delta}(\zeta)$
  for all $\zeta$, as required to finish the proof of Lemma \ref{scale-analysis}.
\end{proof}

\section{The main theorem} \label{mainsectionthm}

\begin{theorem} \label{mainthm} Let $\rho < \kappa < \lambda$ where $\rho$ is regular and uncountable,
  $\lambda$ is the least inaccessible limit of measurable cardinals greater than $\kappa$,
  and there is a Mitchell increasing sequence $\langle E_i : i < \rho \rangle$ such that each
  extender $E_i$ witnesses that $\kappa$ is $\lambda$-strong and
  is such that ${}^\kappa \Ult(V, E_i) \subseteq \Ult(V, E_i)$.
  Then there is a cardinal-preserving generic extension
  in which $\cf(\kappa) = \rho$, $2^\kappa = \lambda$, and $\Spch(\kappa)$ is unbounded in $\lambda$.
\end{theorem}

\begin{proof}  In $V$ let $\mu \in (\kappa, \lambda)$ be measurable in $V$, and let $\theta = 2^\mu = \mu^+$.
  In the generic extension for each $i<\rho$ let $\mu_i = g^*_\mu(i)$ and $\theta_i=g^*_\theta(i)$.
  For all large $i$ we have that in the extension:
  \begin{enumerate}
  \item  $\mu_i$ is measurable and $2^{\mu_i} = \theta_i = \mu_i^+$.
  \item  There is a normal measure $U_i$ on $\mu_i$ generated by an almost decreasing sequence of length
    $\theta_i$.
  \item There exist a cofinal sequence in $\prod_{i < \rho} \mu_i$ under eventual domination of length $\mu$,
    and a cofinal sequence in $\prod_{i < \rho} \theta_i$ under eventual domination of length $\theta$.
  \end{enumerate}

  Appealing to Lemma \ref{scalegeneration}, in the extension there is a uniform ultrafilter
  $U$ on $\kappa$ with $Ch(U) = \theta$.
\end{proof}

\begin{theorem} \label{mainthm2} From the same hypotheses as Theorem \ref{mainthm}, 
  it is consistent that 
  $2^\kappa$ is the least weakly inaccessible cardinal greater than $\kappa$, and
  every regular cardinal between $\kappa$ and $\lambda$ is in the spectrum.
\end{theorem}

\begin{proof} 
  We will force over the model from the proof of Theorem \ref{mainthm} with a suitable
  product of collapsing posets. 
  We enumerate the measurable cardinals in the interval $(\kappa, \lambda)$ as
  $\langle \mu_\eta : \eta < \lambda \rangle$.   
  For every limit $\zeta < \lambda$,  $\sup_{\eta < \zeta} \mu_\eta$
 is singular by the minimality of $\lambda$,
  in particular it is less than $\mu_\zeta$. Now we choose an increasing sequence of regular cardinals $\langle \chi_\eta : \eta < \lambda \rangle$ in
  the interval $(\kappa, \lambda)$ as follows:
  $\chi_0 = \kappa^+$, $\chi_{\eta+1} = \mu_\eta^{++}$ for $\eta < \lambda$, and 
  $\chi_\zeta = (\sup_{\eta < \zeta} \mu_\eta)^+$ when $\zeta$ is a  limit ordinal.
  
  We force with the Easton support product of the Levy collapses $Coll(\chi_\zeta, \mu_\zeta^+)$
 for $\zeta < \lambda$.
  By a routine calculation the surviving cardinals in the interval $(\kappa, \lambda)$ are those
  of the form $\chi_\zeta$ and their limits. All the limits are singular so the regular cardinals in
  $(\kappa, \lambda)$ are those of the form $\chi_\zeta$. Now we argue exactly as
  in \cite[Claim 8 and Theorem 9]{GartiMagidorShelah} that
 $\chi_\zeta$ is in the spectrum for all $\zeta$. 
\end{proof}

\end{document}